\documentclass{amsart}
\usepackage[utf8]{inputenc}
\usepackage{amsmath}
\usepackage{amssymb}
\usepackage{caption}
\usepackage{amsthm}
\usepackage[hidelinks]{hyperref}
\usepackage{cleveref}
\usepackage[justification=centering]{caption}
\crefname{lemma}{Lemma}{Lemmas}
\crefname{theorem}{Theorem}{Theorems}
\usepackage{xcolor}
\usepackage{comment}
\usepackage{graphicx}
\usepackage{tikz}
\usepackage{tikz-cd}
\usepackage{mathrsfs}
\usepackage{enumitem}
\usetikzlibrary{shapes.geometric}
\makeatletter
\def\@settitle{\begin{center}%
  \baselineskip14\p@\relax
  \bfseries
  \uppercasenonmath\@title
  \@title
  \ifx\@subtitle\@empty\else
     \\[1ex]\uppercasenonmath\@subtitle
     \footnotesize\mdseries\@subtitle
  \fi
  \end{center}%
}
\def\subtitle#1{\gdef\@subtitle{#1}}
\def\@subtitle{}
\makeatother

\newtheorem{theorem}{Theorem}[section]
\newtheorem{corollary}{Corollary}[theorem]

\theoremstyle{remark}
\newtheorem{remark}[theorem]{Remark}

\theoremstyle{remark}

\usepackage{chngcntr}
\usepackage{graphicx} 
\usepackage{float}
\counterwithout{equation}{section}
\counterwithout{theorem}{section}

\newcommand{\OO}{\mathrm{O}}
\newcommand{\calO}{\mathcal{O}}
\newcommand{\R}{\mathbb{R}}
\newcommand{\Hyp}{\mathbb{H}}

\begin{document}

\title{Density of systoles of hyperbolic manifolds}

\author{Sami Douba}
\author{Junzhi Huang}

\begin{abstract}
We show that for each $n \geq 2$, the systoles of closed hyperbolic $n$-manifolds form a dense subset of $(0, +\infty)$. We also show that for any $n\geq 2$ and any Salem number $\lambda$, there is a closed arithmetic hyperbolic $n$-manifold of systole $\log(\lambda)$. In particular, the Salem conjecture holds if and only if the systoles of closed arithmetic hyperbolic manifolds in some (any) dimension fail to be dense in $(0, +\infty)$.
\end{abstract}

\address{Institut des Hautes \'Etudes Scientifiques, 35 route de Chartres, 91440 Bures-sur-Yvette, France}
\email{douba@ihes.fr}

\address{Department of Mathematics, Yale University, New Haven, CT 06511, USA}
\email{junzhi.huang@yale.edu}

\maketitle

The {\em systole} $\mathrm{sys}(M)$ of a closed nonpositively curved Riemannian manifold $M$ is the length of a shortest closed geodesic in $M$. 
It follows for instance from residual finiteness of surface groups that there is a sequence of closed hyperbolic surfaces with systole going to infinity. Since for any $\epsilon > 0$ and any closed surface $M$ of negative Euler characteristic there is a hyperbolic metric on $M$ of systole $< \epsilon$, continuity of the systole function on moduli space then yields for any $L>0$ a closed hyperbolic surface of systole precisely $L$.

By Mostow rigidity \cite{mostow1968quasi}, the set of systoles of closed hyperbolic $n$-manifolds is countable for $n \geq 3$. Nevertheless, we show the following.

\begin{theorem}\label{thm:density-of-systoles}
  For any $n\geq 2$, the set of systoles of closed hyperbolic $n$-manifolds is dense in $(0,+\infty)$.
\end{theorem}

That $0$ is an accumulation point of systoles of closed hyperbolic $3$-manifolds follows for instance from Thurston's hyperbolic Dehn filling theory (see \cite[Sections~E.5~and~E.6]{MR1219310}). The analogous fact for $4$-manifolds was established by Agol~\cite{agol2006systoles}, whose strategy was to ``inbreed'' arithmetic manifolds along totally geodesic hypersurfaces. This strategy was successfully extended to arbitrary dimensions independently by Belolipetsky--Thomson~\cite{MR2821431} and Bergeron--Haglund--Wise \cite{bergeron2011hyperplane}. The present note is an elaboration on the utility of this inbreeding technique.

We also establish a variant of Theorem \ref{thm:density-of-systoles} for arithmetic hyperbolic manifolds. A {\em Salem number} is a real algebraic integer $\lambda > 1$ of degree $\geq 4$ that is Galois-conjugate to $\lambda^{-1}$ and all of whose remaining Galois conjugates lie on the unit circle. 

\begin{theorem}\label{thm:arithmetic-systoles}
Let $\lambda > 1$ be a Salem number. Then for any $n \geq 2$, there is a closed arithmetic hyperbolic $n$-manifold of systole precisely $\log(\lambda)$.
\end{theorem}

The manifolds we exhibit in the proof of Theorem \ref{thm:arithmetic-systoles} are  arithmetic of simplest type (the latter being the only arithmetic construction that applies in every dimension) and are moreover {\em classical} in the language of Emery--Ratcliffe--Tschantz \cite{Emery2015SalemNA}. For $n \geq 4$, it follows from Meyer's theorem on quadratic forms and \cite[Thm.~5.2]{Emery2015SalemNA} that the exponential lengths of closed geodesics in any classical simplest-type closed arithmetic hyperbolic $n$-manifold are Salem numbers.\footnote{Exponential geodesic lengths in closed arithmetic hyperbolic surfaces or $3$-manifolds can be {\em quadratic} ``Salem numbers''; the latter are sometimes included in the definition of a Salem number. Exponential geodesic lengths in odd-dimensional closed arithmetic hyperbolic orbifolds need not be Salem numbers even in this more general sense; see Lemma 4.10 and Theorem 4.11 in \cite{MR1184416}, as well as \cite[Thm.~7.7]{Emery2015SalemNA}.}

Lehmer's conjecture for Salem numbers, or the {\em Salem conjecture} for short, asserts that the Salem numbers are bounded away from~$1$; see \cite[Section~13]{MR1503118} and \cite[p.~31]{MR732447}. In light of the relationship between Salem numbers and lengths of geodesics in arithmetic hyperbolic manifolds, Theorem \ref{thm:arithmetic-systoles} yields the following reformulation of the Salem conjecture.

\begin{corollary}
The Salem conjecture holds if and only if, for some, equivalently, any, $n \geq 2$, the set of systoles of closed arithmetic hyperbolic $n$-manifolds fails to be dense in $(0, +\infty)$. 
\end{corollary}
\begin{proof} It is known generally that, given any irreducible symmetric space $X$ of noncompact type, the Salem conjecture implies a uniform lower bound on systoles of closed arithmetic locally symmetric manifolds modeled on $X$; see Fr\c aczyk--Pham~\cite{MR4608431}. On the other hand, since positive powers of Salem numbers remain Salem numbers, failure of the Salem conjecture would imply density of Salem numbers in $(1, +\infty)$, and hence density of systoles of closed arithmetic hyperbolic $n$-manifolds in $(0, +\infty)$ for each $n \geq 2$ by Theorem \ref{thm:arithmetic-systoles}.
\end{proof}

We motivate the proofs of Theorems \ref{thm:density-of-systoles} and \ref{thm:arithmetic-systoles} with the following remark.

\begin{remark}\label{dim2}
We provide an argument along the lines of \cite{agol2006systoles} that, given any $L > 0$, there is a closed hyperbolic surface of systole precisely~$L$. Indeed, it suffices to find a closed hyperbolic surface $M_L$ of systole $\geq L$ possessing two disjoint $2$-sided simple closed geodesics $\Sigma_1$ and $\Sigma_2$ and an orthogeodesic segment $\omega$ joining the $\Sigma_i$ of length~$L/2$ and whose length is minimal among all orthogeodesic segments with endpoints on $\Sigma_1 \cup \Sigma_2$. Cutting $M_L$ along the $\Sigma_i$ and then doubling the resulting surface along its boundary yields a (possibly disconnected) closed hyperbolic surface of systole precisely $L$. 

One can construct such a surface $M_L$ as follows. Let $P \subset \mathbb{H}^2$ be a right-angled pentagon with an edge $\widetilde{\omega}$ of length precisely $L/2$, and let $H_1, H_2 \subset \mathbb{H}^2$ be the walls of $P$ adjacent to $\widetilde{\omega}$. Let $\Gamma_P < \mathrm{Isom}(\mathbb{H}^2)$ be the group generated by the reflections in the walls of $P$. We can find a larger right-angled convex polygon $Q \subset \mathbb{H}^2$ that is a union of finitely many $\Gamma_P$-translates of $P$ such that the walls of $Q$ that enter the $\frac{L}{2}$-neighborhood of $H_i$ are orthogonal to $H_i$ for $i =1,2$ (this idea originates in work of Scott \cite{MR494062}; see also \cite[Section~3.1]{MR1836283}). By residual finiteness and virtual torsion-freeness of the group $\Gamma_Q < \mathrm{Isom}(\mathbb{H}^2)$ generated by the reflections in the walls of $Q$, there is a finite-index subgroup $\Lambda <\Gamma_Q$ such that $\Lambda \backslash \mathbb{H}^2$ is a surface of systole $\geq L$. We may now take $M_L = \Lambda \backslash \mathbb{H}^2$, the $\Sigma_i$ to be the projections to $M_L$ of the $H_i$, and $\omega$ to be the projection of $\widetilde{\omega}$.
\end{remark}

\begin{proof}[Proof of Theorem \ref{thm:density-of-systoles}]

Given $L, \epsilon>0$, we  exhibit a closed hyperbolic $n$-manifold $M$ with $\lvert\text{sys}(M)-L\rvert<\epsilon$. As in Remark \ref{dim2}, it suffices to find a closed hyperbolic $n$-manifold $M_{L, \epsilon}$ of systole $\geq L + \epsilon$ possessing two disjoint $2$-sided closed embedded totally geodesic hypersurfaces $\Sigma_1$ and $\Sigma_2$, and an orthogeodesic segment $\omega$ joining the $\Sigma_i$ of length within $\epsilon/2$ from $L/2$ and whose length is minimal among all orthogeodesic segments with endpoints on $\Sigma_1 \cup \Sigma_2$. We can then cut $M_{L, \epsilon}$ along $\Sigma_1\cup\Sigma_2$, select the component $N$ containing $\omega$ of the resulting manifold with boundary, and define $M$ to be the double of $N$. The double of $\omega$ is indeed a shortest closed geodesic in $M$ of length within $\epsilon$ from  $L$.

To that end, let $k=\mathbb{Q}(\sqrt{2})$ and $\calO_k = \mathbb{Z}[\sqrt{2}]$. Define $f$ to be the quadratic form on $\R^{n+1}$ given by
\[f(x_1,\cdots,x_{n+1})=x_1^2+\cdots+x_n^2-\sqrt{2}x_{n+1}^2.\]

We identify the level set $\{x\in\R^{n+1}\vert f(x)=-1, x_{n+1}>0\}$ with $n$-dimensional hyperbolic space $\Hyp^n$ and $\OO'(f;\R)$ with $\mathrm{Isom}(\Hyp^n)$, where $\OO'(f;\R)$ is the index-$2$ subgroup of $\OO(f;\R)$ preserving $\Hyp^n$. By the Borel--Harish-Chandra theorem \cite{MR147566}, we have that $\Gamma:= \OO'(f;\calO_k)$ is a uniform arithmetic lattice of $\mathrm{Isom}(\Hyp^n)$. 

Denote by $H_1$ the hyperplane $\{x_1=0\}$ in $\Hyp^n$. Since $\OO'(f;k)$ is dense in $\OO'(f;\R)$, there is an element $g\in\OO'(f;k)$ such that
\[\left\lvert\mathrm{dist}_{\Hyp^n}(H_1,gH_1)-\frac{L}{2}\right\rvert < \frac{\epsilon}{2}.\] 
It follows again from the Borel--Harish-Chandra theorem that $H_1$ and $H_2:= gH_1$ project to closed immersed totally geodesic hypersurfaces in the orbifold $\Gamma \backslash \Hyp^n$.
Let $\widetilde{\omega}$ be the orthogeodesic segment in $\Hyp^n$ connecting $H_1$ and $H_2$. By the proof of Lemma 3.1 in Belolipetsky--Thomson  \cite{MR2821431}, we may now pass to a non-zero ideal $I \subset \mathcal{O}_k$ such that $\mathrm{dist}_{\Hyp^n}(H_1, \gamma H_2) \geq \mathrm{dist}_{\Hyp^n}(H_1, H_2)$ for each $\gamma$ in the principal congruence subgroup $\Gamma(I)<\Gamma$ of level $I$.
Up to diminishing the ideal $I$, we can further assume that
\begin{enumerate}
\item the principal congruence subgroup $\Gamma(I)$ is torsion-free, so that $M_{L, \epsilon}:= \Gamma(I) \backslash \Hyp^n$ is a manifold;
\item\label{rf} we have $\mathrm{sys}(M_{L, \epsilon}) \geq L + \epsilon$;
\item\label{separability} for $i=1,2$ and $\gamma \in \Gamma(I)$, either $\gamma H_i = H_i$, in which case $\gamma$ furthermore preserves each side of $H_i$, or $\mathrm{dist}_{\Hyp^n}(H_i, \gamma H_i) \geq \mathrm{dist}_{\Hyp^n}(H_1, H_2)$; in particular, the projections $\Sigma_i$ of the $H_i$ to $M_{L, \epsilon}$ are $2$-sided closed {\em embedded} totally geodesic hypersurfaces in $M_{L, \epsilon}$.

\end{enumerate}

Note that item (\ref{rf}) can be arranged by residual finiteness of the ring $\mathcal{O}_k$, whereas item (\ref{separability}) can be ensured by the fact that the stabilizer in $\Gamma$ of either side of $H_i$ is an intersection of congruence subgroups of $\Gamma$ for $i=1,2$; see \cite{MR0898729} and \cite[Lemme~principal]{MR1769939}. We now have that the $\Sigma_i$ are disjoint and that any orthogeodesic segment in $M_{L, \epsilon}$ with endpoints on $\Sigma_1 \cup \Sigma_2$ has length at least that of the projection $\omega$ of $\widetilde{\omega}$ to $M_{L,\epsilon}$. Thus, the manifold $M_{L, \epsilon}$, the hypersurfaces $\Sigma_i$, and the orthogeodesic segment $\omega$ are as desired.
\end{proof}

\begin{remark}\label{mahler}
The proof of Theorem \ref{thm:density-of-systoles} demonstrates that we may take the closed hyperbolic manifolds giving rise to a dense set of systoles in $(0, +\infty)$ to be quasi-arithmetic in the sense of Vinberg and all share the same adjoint trace field (indeed, the same Vinberg ambient algebraic group); see \cite{MR3451458} and compare \cite[Thm.~1.3]{MR4557016}. In this case, infinitely many of these manifolds would necessarily be nonarithmetic. This is because, for fixed $D >0$ and $\mu > 1$, there are only finitely many monic integer polynomials of degree $\leq D$ and Mahler measure $\leq \mu$; see the discussion immediately following Conjecture 10.2 in \cite{MR2084613}. 
 On the other hand, by varying the form $f$ in the proof, one easily produces a family of pairwise incommensurable closed hyperbolic manifolds in each dimension whose systoles remain dense in $(0, +\infty)$. 

In fact, it is possible to achieve the latter without varying $f$; one can select from any family of closed hyperbolic manifolds whose systoles are dense in $(0, +\infty)$ pairwise incommensurable manifolds whose systoles remain dense in $(0, +\infty)$. Indeed, given a closed hyperbolic manifold $M$, the set of all geodesic lengths of manifolds commensurable to $M$ is closed and discrete in $[0, +\infty)$. When $M$ is nonarithmetic, the latter follows from Margulis's arithmeticity criterion \cite[Thm.~IX.6.5]{margulis91discrete}, and when $M$ is arithmetic, from the fact about Mahler measures of bounded-degree monic integer polynomials mentioned in the previous paragraph.
\end{remark}

\begin{remark}
Instead of appealing to the work of Belolipetsky--Thomson (loc. cit.) to find the manifold $M_{L, \epsilon}$ in the proof of Theorem \ref{thm:density-of-systoles}, we could have instead used \cite[Cor.~1.12]{bergeron2011hyperplane} as in the proof of Theorem 4 in \cite{douba2023systoles}. An interesting feature of the former approach is that it afforded us a manifold $M_{L, \epsilon}$ that is {\em congruence} arithmetic.
\end{remark}

\begin{proof}[Proof of Theorem \ref{thm:arithmetic-systoles}]
Let $\mu = \lambda + \lambda^{-1}$. Let $k= \mathbb{Q}(\mu)$, and $\mathcal{O}_k$ be the ring of integers of $k$. Following the proof of \cite[Thm.~6.3]{Emery2015SalemNA}, we define the form $f$ in $n+1$ variables over $k$ to be that given by the symmetric matrix
\[
\begin{pmatrix}
1 & & \mu/2 \\
 & I_{n-1} & \\
 \mu/2 & & 1
\end{pmatrix}.
\]
The form $f$ is of signature $(n,1)$, so that we may again identify $\Hyp^n$ with one of the two sheets of the hyperboloid $\{x\in\R^{n+1}\vert f(x)=-1\}$ and $\mathrm{Isom}(\Hyp^n)$ with $\OO'(f;\R)$, where $\OO'(f;\R)$ is the index-$2$ subgroup of $\OO(f;\R)$ preserving $\Hyp^n$. Moreover, the field $k$ is totally real, and for any embedding $\sigma:k\to\R$ other than the identity (of which at least one exists), the form $f^\sigma$ is positive definite. It follows that $\Gamma := \OO'(f ; \mathcal{O}_k)$ is a uniform arithmetic lattice of $\OO'(f ; \mathbb{R})= \mathrm{Isom}(\Hyp^n)$, which contains the reflections
\[
\tau_1 := \begin{pmatrix}
-1 & & -\mu \\
 & I_{n-1} & \\
  & & 1
\end{pmatrix}, \> \>
\tau_2 := \begin{pmatrix}
-\mu & & 1-\mu^2 \\
 & I_{n-1} & \\
 1 & & \mu
\end{pmatrix}
\]
whose respective fixed hyperplanes $H_1, H_2 \subset \Hyp^n$ satisfy \[\mathrm{dist}_{\Hyp^n}(H_1, H_2) = \log(\lambda)/2.\]  As in the proof of Theorem \ref{thm:density-of-systoles}, by the aforementioned work of Belolipetsky--Thomson (loc. cit.), we may pass to a non-zero ideal $I \subset \mathcal{O}_k$ such that
\begin{itemize}
\item the principal congruence subgroup $\Gamma(I)$ is torsion-free, so that $M_{\log(\lambda)}:= \Gamma(I) \backslash \Hyp^n$ is a manifold;
\item we have $\mathrm{sys}(M_{\log(\lambda)}) \geq \log(\lambda)$;
\item the $H_i$ project to disjoint closed embedded $2$-sided totally geodesic hypersurfaces $\Sigma_1$ and $\Sigma_2$ in $M_{\log(\lambda)}$; 
\item any orthogeodesic segment in $M_{\log(\lambda)}$ with endpoints on $\Sigma_1 \cup \Sigma_2$ has length $\geq \log(\lambda)/2$. 
\end{itemize}
Now cutting $M_{\log(\lambda)}$ along the $\Sigma_i$ and then doubling the resulting manifold along its boundary yields a (possibly disconnected) closed manifold $M$ of systole precisely $\log(\lambda)$ each of whose components is arithmetic. 
\end{proof}

\begin{remark}
Note that each component of the output manifold $M$ in the proof of Theorem \ref{thm:arithmetic-systoles} is arithmetic because the reflections $\tau_i$ in the hyperplanes $H_i$ both belong to the input arithmetic lattice $\Gamma$, whereas there is no such guarantee in the proof of Theorem \ref{thm:density-of-systoles}.
\end{remark}

We conclude by remarking that there are natural spectral analogues of the questions considered in this paper, namely, what are the possible values of the smallest nonzero eigenvalue of the Laplacian of a closed hyperbolic manifold (respectively, an arithmetic hyperbolic manifold)? For recent progress in dimension $2$, where these questions are already nontrivial, see Kravchuk--Maz\'a\v{c}--Pal \cite{Kravchuk_2024}  and Magee \cite{magee2024limitpointsbassnotes}. See also Breuillard--Deroin \cite{MR4216688} for a spectral reformulation of the Salem conjecture.

\subsection*{Acknowledgements} We thank Misha Belolipetsky, Bram Petri, Alan Reid, and the anonymous referee for helpful comments. S.D. was supported by the Huawei Young Talents Program. J.H. was partially supported by NSF grant DMS-2005328. This work was initiated during a visit of S.D. to Yale University in March 2024; S.D. thanks Sebastian Hurtado and the entire geometry/topology/dynamics team at Yale for their hospitality.

\bibliography{densitybib}{}
\bibliographystyle{siam}

\end{document}